\documentclass[12pt]{amsart}
\usepackage{amsfonts,amsmath,amscd}
\usepackage{epsfig,graphics}
\usepackage{amssymb}
\usepackage{amscd}
\usepackage{hyperref}
\usepackage{bbm}
\usepackage{mathabx}
\usepackage{xcolor}

\newcommand{\BN}{{\mathbb{N}}}
\newcommand{\BR}{{\mathbb{R}}}
\newcommand{\BC}{{\mathbb{C}}}

\newcommand{\BP}{{\mathbb{P}}}

\newcommand{\BE}{{\mathbb{E}}}

\newcommand{\gd}{\delta}
\newcommand{\gb}{\beta}

\newcommand{\Ga}{\Gamma} 

\newcommand{\gO}{\Omega}

\newcommand{\gl}{\lambda}
\newcommand{\ga}{\alpha}

\newcommand{\dd}{{\partial}}

\def\vp{\varphi}
\def\e{\varepsilon}

\newcommand{\ol}[1]{\overline{#1}}

\newcommand{\tprob}{\text{Prob}}

\newcommand{\diam}{\text{diam}}

\newcommand{\df}{\stackrel{\text{def}}{=}}

\newcommand{\half}{\frac{1}{2}}

\newtheorem{prop}{Proposition}[section]
\newtheorem*{prop*}{Proposition}
\newtheorem{thm}{Theorem}

\newtheorem*{thm*}{Theorem}
\newtheorem{lem}[prop]{Lemma}
\newtheorem{cor}[prop]{Corollary}
\newtheorem*{cor*}{Corollary}

\theoremstyle{definition}

\newtheorem{defn}[prop]{Definition}
\newtheorem*{defn*}{Definition}
\newtheorem{rem}[prop]{Remark}
\newtheorem*{rem*}{Remark}
\newtheorem{exam}[prop]{Example}

\title[Special affine representations for hyperbolic groups.]{On the special affine representations for hyperbolic groups.}

\begin{document}
\maketitle
\centerline{\scriptsize KEVIN BOUCHER}
\centerline{\tiny University of Southampton, Highfield, Southampton SO17 1BJ, Great Britain}
\centerline{\tiny (e-mail: k.l.boucher@soton.ac.il)}

\begin{abstract}
In this paper we extend the construction of special representations \cite{MR2328257} \cite{MR2000965} to Gromov hyperbolic groups which admits complementary series in the sense of \cite{Boucher:2020aa}.
We prove that these representations have a natural non-trivial reduced cohomology class $[c]$.
An analogue of Kuhn-Vershik's formula \cite{MR2000965} is established and as a by-product a characterisation of hyperbolic groups that admit complementary series.
Investigating dynamical properties of the cohomology class $[c]$ we prove an cocycle equidistribution theorem \'a la Roblin-Margulis \cite{MR2057305} \cite{MR2035655} (see also  \cite{Garncarek:2014aa} \cite{MR3939562} \cite{Boucher:2020ab}) and deduce the irreducibility of the associated affine actions \cite{Neretin:aa} \cite{MR3549540} \cite{MR3711276}.
The irreducibility of the affine actions associated to the canonical class $[c]$ is original even in the case of uniform lattices in $SO(n,1)$, $SU(n,1)$ or $SL_2(\mathbb{Q}_p)$ with $n\ge 1$ and $p$ prime.
\end{abstract}

{\tiny \textbf{Key words:} Affine actions, conditionally negative functions, complementary series, harmonic cocycles, Knapp-Stein operators.}

\section{Introduction}
Even though most discrete groups  do not admit interesting isometric affine actions on Hilbert spaces \cite{MR1408975} some of those acting sufficiently well on negatively curved spaces appear to have interesting ones \cite{MR578893} \cite{MR3309999} \cite{MR2328257}.
Given a rank 1 Lie group $G$  and $X=G/P$ its symmetric space  where $P$ stands for a minimal parabolic subgroup, a remarkable example can be constructed from the natural action of $G$ on its top degree forms with zero average, $\gO_0^\text{top}(X)$.
This space can be endowed with a natural invariant quadratic form $Q'$ and a canonical proper cocycle (for the form $Q'$) associated to the extension:
\begin{equation}\label{eq:num1}
0\rightarrow \gO_0^\text{top}(X)\rightarrow \gO^\text{top}(X)\rightarrow \BC\rightarrow 0
\end{equation}
When $G=SO_o(n,1)$ or $SU(n,1)$ it was proved in \cite{MR245725} that this quadratic form is positive definite and the representation obtained might be seen as a limit of complementary series in an appropriate sense \cite{MR2328257}.
As shown in \cite[Chap. 5]{MR578893}, this representation have a particular role in cohomology theory of rank 1 Lie groups.

\begin{rem}
In the case of rank 1 Lie groups with the Kazhdan property (T) \cite{MR2415834} such as $G=Sp(n,1)$ for $n\ge 2$, a Hilbert structure on $\gO_0^\text{top}(X)$ inducing a uniformly bounded representation of $G$ with non-trivial cohomology was built  in \cite{Nishikawa:2020aa}. 
\end{rem}

The main purposes of this work are the extension of this construction to the class of Gromov hyperbolic groups and the investigation of its intrinsic properties and relations with the boundary complementary series representations introduced in \cite{Boucher:2020aa}:

\subsection{Statement of the results.}\label{subsec:intro:res}\hfill\break
The reader can refer to Section \ref{sec:prelim} for details about the definitions.\\
Let $(\Ga,d_\Ga)$ be a discrete finitely generated Gromov hyperbolic endowed with a strong hyperbolic distance quasi-isometric to a word distance.
The group $\Ga$ is supposed to be non-elementary, i.e non-virtually abelian and its critical exponent is denoted $\gd$.
Let $(\dd \Ga,d)$ be its Gromov boundary endowed with a $\Ga$-conformal visual distance induced by $d_\Ga$, $D$ the Hausdorff dimension of $(\dd \Ga,d)$ and $\nu=\nu_e$ the Patterson-Sullivan measure at $e\in \Ga$.

\begin{rem}
As a first approximation one can think of $(\Ga,d_\Ga)$ as a Gromov hyperbolic group endowed with a word distance.
\end{rem}

Given $s\in\BR$ the space of {\it $s$-densities }refers to the linear representation by bounded operators on the space of continuous function, $C(\dd \Ga)$, defined as:
$$\pi_s(g)\vp(\xi)=e^{-s\gd b_\xi(g.o,o)}\vp(g^{-1}\xi)$$
where $b$ stands for the Busemann cocycle, $g\in \Ga$, $\vp\in C(\dd \Ga)$ and $\xi\in\dd \Ga$.
Those representations extend to $L^2(\dd \Ga,\nu)$ as linear actions by bounded operators with $\pi_s^*=\pi_{1-s}$.\\
For $s>\half$ the {\it Knapp-Stein operator} $I_s$ on $L^2(\dd \Ga, \nu)$ is the operator defined as:
$$I_s:L^2(\dd \Ga, \nu)\rightarrow L^2(\dd \Ga, \nu),\quad \vp\mapsto \int_{\dd \Ga}\frac{\vp(\xi)}{d^{2D(1-s)}(\xi,.)}d \nu(\xi)$$
The operator $I_s$ with $s>\half$ is well defined, intertwines the representation $\pi_s$ with $\pi_{1-s}$ and thus induces a unitary representation whenever it is positive definite as an operator on $L^2(\dd \Ga, \nu)$.\\

Analogously to the rank 1 case we construct another potential operator, $I'$, on $L^2(\dd \Ga,\nu)$ as:
$$I':L^2(\dd \Ga, \nu)\rightarrow L^2(\dd \Ga, \nu),\quad \vp\mapsto \int_{\dd \Ga}\vp(\xi)(\xi,.)d \nu(\xi)$$
The operator $I'$ can be thought as the infinitesimal generator of the $(I_s)_{0< s\le 1}$ and induces an intertwiner between the $\pi_1$-sub-representation on $L^2_0(\dd \Ga, \nu)\df L^2(\dd \Ga, \nu)\ominus\BC$ and its dual the $\pi_0$-representation on $L^2(\dd \Ga, \nu)/\BC$.
In particular $I'$ defines a positive and $\pi_1$-invariant quadratic form on the space of 1-densities with zero average whenever $I'$ is positive as an operator on $L^2_0(\dd \Ga, \nu)$ analogously as the quadratic form $Q'$ mentioned earlier. \\
Moreover an exact sequence similar as \ref{eq:num1} produces a canonical $\Ga$-equivariant cocycle with coefficients in $\pi_1$:
$c:\Ga\times \Ga\rightarrow L^2_0(\dd \Ga, \nu)$ with the following property:

\begin{thm}\label{thm:main1}
Assume the operator $I'$  is positive definite on $L^2_0(\dd \Ga, \nu)$.
Then the associated unitary representation, $(\pi_1,H_1')$, is mixing and $c$ has non-trivial reduced cohomology class.
\end{thm}
Following a similar terminology as in \cite{MR2328257} we call $(\pi_1,H_1')$ and its dual $(\pi_0,H_0')$ (see Corollary \ref{cor:dual'}) {\it special representations}.\\

A $\Ga$-invariant distance, $d$, on $\Ga$ is called {\it $\Ga$-roughly conditionally negative} if there exist a isometric affine action on a Hilbert space $H$ of cocycle $b:\Ga\rightarrow H$, $t>0$ and $C\ge 0$ such that $|t.d(g,h)-\|b(g)-b(h)\|_H^2|\le C$ for all $g,h\in \Ga$.

\begin{exam}
The ambient distance on a uniform lattice, $\Ga$, in $SO(n,1)$ or $SU(m,1)$ for $n\ge3$ and $m\ge 2$ gives an example of $\Ga$-roughly conditionally negative (see \cite[Thm. 3.9]{MR2346274}).
\end{exam}

\begin{thm*}[\cite{Boucher:2020aa} Theorem 1]
If the distance $d_\Ga$  is \textit{$\Ga$-roughly conditionally negative}, then the operators $I_s$ with $s>\half$ are positive on $L^2(\dd \Ga, \nu)$.
\end{thm*}
It appears that this result is tight.
Indeed an explicit computation of the cocycle's norm given by Theorem \ref{thm:main1} gives an analogue of Kuhn Vershik's formula \cite[Thm. 5]{MR2000965} and as a by-product allows us to characterise geometrically the existence of boundary complementary series representations using  \cite[Thm. 1]{Boucher:2020aa}:

\begin{thm}\label{thm:maincor}
The following properties are equivalent:
\begin{enumerate}
\item the distance $d_\Ga$ on  $\Ga$ is $\Ga$-roughly conditionally negative;
\item the family of Knapp-Stein operators $(I_s)_s$ over $L^2(\dd \Ga, \nu)$ are positive for all $s>\half$ and thus $\Ga$ has complementary series over the all interval $(\half,1]$ in the sense of \cite{Boucher:2020aa};
\item the operator $I'$ is positive over $L^2_0(\dd \Ga, \nu)$.
\end{enumerate}
\end{thm}

Even if every a-T-menable group acts properly on a infinite dimensional real hyperbolic space by isometries \cite{MR1253543} (which has conditionally negative distance), it is a open question whether all a-T-menable Gromov hyperbolic groups act geometrically on a proper Gromov hyperbolic space with a roughly conditionally negative distance in the above sense.\\

The cocycle $c$ enjoy similar equidistribution properties as the ones observed for the boundary representations $(\pi_s)_{s\in [0,1]}$ in \cite{Garncarek:2014aa} \cite{Boucher:2020aa} \cite{Boucher:2020ab} \cite{MR3939562}.
This led us to the following cocycle analog of the Roblin-Margulis equidistribution theorem:
\begin{thm}\label{thm:main2}
Let $(\nu_n)_n$ be a spherical approximation of the Patterson-Sullivan probability $\nu$ as in Proposition \ref{prop:equid} and $(\pi_1,H_1')$ the special representation.
Then for all $u,w\in H_1'$ and $f\in C(\ol{\Ga})$ with $ \nu(f)=0$ one has 
$$\lim_n\half\sum_\Ga f(g)(I'(c(g)+(w-\pi_1(g)w)),u)\nu_n(g)=(I'([f|_{\dd \Ga}]),u)$$
where $[f|_{\dd \Ga}]$ stands for the class of $f|_{\dd \Ga}$ in $H_1'$.
\end{thm}

\begin{rem}
Not every element in $H_1'$ can be represented by genuine function on $\dd \Ga$. In particular the use of $(I'(u),w)$ for $u,v\in H_1'$ to denote the scalar product on $H_1'$ is an abuse of notation but as justified in Appendix \ref{app:banach} $I'$ and the $L^2$-scalar product naturally extend respectively as a isometry and a coupling between $H_1'$ and its contragredient dual $H_0'$ (see Corollary \ref{cor:dual'}). 
\end{rem}

As introduced by Neretin in \cite{Neretin:aa} and investigated in \cite{MR3549540} and \cite{MR3711276}, \textit{an isometric $\Ga$-affine action on a Hilbert space $H$ is affine irreducible (AI) if  $H$ is the only closed affine subspace $\Ga$-stable.}

As a consequence of Theorem \ref{thm:main2} we obtain:
\begin{thm}\label{thm:main3}
The affine commutant of the $\Ga$-affine action $(\pi_1,H_1',c)$ is trivial.
In particular $(\pi_1,H_1',c)$ is irreducible.
\end{thm}

\subsection{Organisation}$ $\\
The Section \ref{sec:prelim} is dedicated to standard material about Gromov hyperbolic geometry and Patterson-Sullivan theory.
We construct the special representations in Section \ref{sec:sperep}.
The Section \ref{sec:cohomology} is dedicated to cohomological aspects of these representations where Theorem \ref{thm:main1} and \ref{thm:maincor} are proved.
We conclude by Section \ref{sec:affine} and prove the equidistribution Theorem \ref{thm:main2} together with Theorem \ref{thm:main3}.

\subsection{Acknowledgments}$ $\\
This work was supported by the Weizmann institute of science.

\subsection{Notations and conventions}$ $\\

Given two real valued functions, $a,b$, over a set $S$, we write $a\prec b$ if there exists $C>0$ such that $a(s)\le Cb(s)$ for all $s\in S$ and $a\asymp b$ if $a\prec b$ and $b\prec a$.
Similarly we write $a\lesssim b$ if there exists $c$ such that $a(s)\le b(s)+c$ and $a\simeq b$ if $a\lesssim b$ and $b\lesssim a$.


\section{preliminaries}\label{sec:prelim}

\subsection{Hyperbolic spaces and groups}\hfill \break
The reader can refer to \cite{MR1086648} \cite{Kapovich:aa} and \cite[Chap. H]{MR1744486} for more about Gromov hyperbolic geometry.\\

Let $(X,d_X)$ be a proper metric space and $o\in X$ a fixed basepoint.\\
The space $(X,d_X)$ is \textit{$K$-roughly geodesic} for $K\ge 0$ if for all $x,y\in X$, there exist a interval $I\subset\BR$ and a map $c:I\rightarrow X$ such that
$|d_X(c(t_1),c(t_2))-|t_1-t_2||\le K$ for all $t_1,t_2\in I$.
The space $(X,d_X)$ is \textit{roughly geodesic} if $(X,d_X)$ is \textit{$K$-roughly geodesic} for some $K\ge 0$.\\
The \textit{Gromov product} of $(X,d_X)$ at $o\in X$ is defined as
$$(x,y)_o\df\half(d_X(x,o)+d_X(y,o)-d_X(x,y))$$
for $x,y\in X$ and the space $(X,d_X)$ is \textit{Gromov hyperbolic} if there exists $C\ge0$ such that
\begin{equation}\label{eq:gromov:product}
(x,y)_o\ge \min\{(x,z)_o,(z,y)_o\}-C
\end{equation}
for all $x,y,z\in X$.

Assuming $(X,d_X)$ is Gromov hyperbolic the \textit{Gromov boundary} of $(X,d_X)$ can be defined as 
$$\dd X\df\{(x_n)_n\in X^\BN:\liminf_{n,m}(x_n,x_m)_o=+\infty\}/\sim$$
with $(x_n)_n\sim(y_n)_n\Leftrightarrow\liminf_{n,m}(x_n,y_m)_o=+\infty$.
The space $\dd X$ endowed with the {\it visual topology} is compact as the space $(X,d_X)$ is proper. Moreover there exists a compatible topology on $\ol{X}=X\cup\dd X$ that makes it a compact space and every pair of points in $\ol{X}$ can be connected by a rough geodesic \cite[Chap. H]{MR1744486}.

\begin{defn}[\cite{MR3551185}]\label{def:str:hyp}
A hyperbolic metric space $(X,d_X)$ is strongly hyperbolic if there exists $\e_0>0$ such that
$$e^{-\e (x,y)_o}\le e^{-\e (x,z)_o}+e^{-\e (z,y)_o}$$
for all $x,y,z,o\in X$ and $0<\e\le \e_0$.
\end{defn}

\begin{exam}
Every CAT(-1) spaces are strongly hyperbolic \cite{MR3551185}.
\end{exam}
Assuming $(X,d_X)$ is strongly hyperbolic, the Gromov product extends continuously to the Gromov compactification $\ol{X}=X\cup \dd X$ (with $(\xi,\xi)=+\infty$ for all $\xi\in\dd X$) and a {\it visual distance} $d$ on $\dd X$ can be defined by $d(\xi,\eta)\df e^{-\e_0 (\xi,\eta)_o}$ for all $\xi,\eta\in \dd X$ where $\e_0>0$ is chosen as in Definition \ref{def:str:hyp}.
Strong hyperbolicity also implies that $\lim_{z\rightarrow \xi} d_X(x,z)-d_X(y,z)=b_\xi(x,y)=d_X(x,y)-2(x,\xi)_y$ is well defined for all $\xi\in\dd X$ and $x,y\in X$.\\

A discrete finitely generated group, $\Ga$, is Gromov hyperbolic if $\Ga$ endowed with a word distance coming from a finite generating set is Gromov hyperbolic as a metric space.
\begin{thm*}[Mineyev-Yu \cite{MR1914618} ]
Every Gromov hyperbolic group admits an invariant roughly geodesic and strongly hyperbolic distance on themselves that is quasi-isometric to a word distance.
\end{thm*}
See \cite{MR2346214} and \cite{MR3551185} for more on the subject.

\begin{exam}
If $\Ga$ is convex-cocompact subgroup of isometries of a roughly geodesic strongly hyperbolic spaces $X$.
Then the pull-back distance on an orbit of $\Ga$ induces a roughly geodesic and strongly hyperbolic distance that is quasi-isometric to a word distance.
\end{exam}

Given a group $\Ga$ acting by isometries on a strongly hyperbolic space $(X,d_X)$, the following conformal relation holds due to the continuous extension of the Gromov product on the boundary:
\begin{equation}\label{eq:conf:fond}
(gx,gy)_o=\half b_x(go,o)+\half b_y(go,o)+(x,y)_o
\end{equation}
for all $x,y\in \ol{ X}=X\cup\dd X$ and $g\in \Ga$.

Also:
\begin{align*}
d(g\xi,g\eta)&\df e^{-\e_0(g\xi,g\eta)_o}\\
&=e^{\half\e_0b_\xi(go,o)}e^{\half\e_0b_\eta(go,o)}e^{-\e_0(\xi,\eta)_o}\\
&=e^{\half\e_0b_\xi(go,o)}e^{\half\e_0b_\eta(go,o)}d(\xi,\eta)
\end{align*}
 for any $\xi,\eta\in \dd X$ and $g\in \Ga$, in other words $d$ is $\Ga$-conformal (see \cite{MR2346214}).

\subsection{The Patterson-Sullivan theory and equidistribution}\label{sec:disin}\hfill\break

Thoughout the rest $(\Ga,d_\Ga)$ denote a discrete Gromov hyperbolic group endowed with a roughly geodesic and strongly hyperbolic distance quasi-isometric to a word distance.
The group $\Ga$ is assumed to be non-virtually abelian.
The Gromov product at $e\in \Ga$ between $x,y\in \ol{\Ga}=\Ga\cup \dd \Ga$ is denoted $(x,y)$ and  the visual parameter $\e_0$ in Definition \ref{def:str:hyp} is simply denoted $\e$.

\begin{defn}
The critical exponent of $\Ga$ is defined as the infimum over all $t\ge0$ such that the Poincar\'e integral at $s$:
$$P_t\df \sum_\Ga e^{-sd_\Ga(g,e)}$$
is finite and denoted $\gd\in(0,+\infty)$.
\end{defn}

It follows from \cite[Cor. 7.3]{MR1214072} that $\Ga$ has finite critical exponent and is of divergent type, i.e. its Poincar\'e series, $P_t$, diverges at $t=\gd$.

\begin{defn}
A $\ga$-conformal density, for some positive $\ga$, is a continuous $\Ga$-equivariant map:
$$\nu_.:\Ga\mapsto \tprob(\dd \Ga);\quad g\mapsto \nu_g$$
such that for all $g,h\in \Ga$, $\nu_g\sim\nu_h$ and
$$\frac{d\nu_g}{d\nu_h}(\xi)=e^{-\ga b_\xi(g,h)}$$
for $[\nu_.]$-almost every $\xi\in\dd \Ga$ and $R\ge R_0$.
\end{defn}

As a Corollary of \cite[Thm. 7.7]{MR1214072} together with strong hyperbolicity:
\begin{thm*}
There exists an unique $\gd$-conformal density on $\Ga$.
Moreover the probability measures $(\nu_g)_g$ are ergodic and satisfy the Shadow Lemma:
$$\nu_g(\{\xi:(\xi,g)\ge d_\Ga(g,e)-R\})\asymp_g e^{\gd R}e^{-\gd d_\Ga(g,e)}$$
\end{thm*}
In the rest we denote $\nu$ the density $\nu_e$.\\

As consequence of the Shadow Lemma the measure $\nu$ is Ahlfors regular, that is:
$\nu(B(\xi,r))\asymp r^{D}$ for all $r\in[0,\diam(\dd \Ga,d)]$ and $D$ the Hausdorff dimension of $(\dd \Ga,d)$ (see \cite[Prop. 7.4]{MR1214072}).
In particular the critical exponent $\gd$ and the Hausdorff dimension $D$ are related by $\e D=\gd$ where $\e$ stands for the visual parameter in Definition \ref{def:str:hyp}.\\

Given $0\le\ga<D$ and $\xi\in\dd \Ga$, as a consequence of Ahlfors regularity one has:
\begin{align*}
\int_{\dd \Ga} e^{-\e\ga(\xi,\eta)}d\nu(\eta)&=\ga\int_0^\infty t^{\ga-1}\nu(\{e^{-\e(\xi,.)}<t^{-1}\})dt\\
&\asymp\ga\int_{\text{Diam}(\dd \Ga,d)^{-1}}^\infty t^{\ga-D-1}dt
+\ga\int_0^{\text{Diam}(\dd \Ga,d)^{-1}} t^{\ga-1}dt\\
&=\frac{\ga}{D-\ga}\text{Diam}(\dd \Ga,d)^{D-\ga}+\text{Diam}(\dd \Ga,d)^{-\ga}
\end{align*}
is finite.
In particular 
$$\sup_{\xi\in \dd \Ga}\int_{\dd \Ga} e^{\e\ga(\xi,\eta)}d\nu(\eta)<+\infty$$
for all $0\le\ga<D$.

\begin{lem}\label{lem:proj:inf}
There exists $L>0$, for all $g\in \Ga$, there exists $\xi\in \dd \Ga$ such that $d_\Ga(g,e)-L\le (g,\xi)\le d_\Ga(g,e)$.
\end{lem}
\begin{proof}
Observe that $(g,z)+(g^{-1},g^{-1}z)=d_\Ga(g,e)$ for all $z\in \ol{\Ga}$
and thus for any pair $\xi\neq\eta$ in $\dd \Ga$:
$$\max\{(g,g\xi),(g,g\eta)\}=d_\Ga(g,e)-\min\{(g^{-1},\xi),(g^{-1},\eta)\}\ge d_\Ga(g,e)-(\xi,\eta)-C$$
where $C$ comes from Equation \ref{eq:gromov:product}.
It is enough to take $L=C+(\xi,\eta)$.
\end{proof}

\begin{cor}\label{cor:unif:power}
$$\sup_{z\in \ol{ \Ga}}\int_{\dd \Ga} e^{\e\ga(z,\eta)}d\nu(\eta)<+\infty$$
for all $0\le\ga<D$.
\end{cor}
\begin{proof}
Let $L$ as in Lemma \ref{lem:proj:inf}.
Given $g\in \Ga$, let $\xi_g$ such that 
$$d_\Ga(g,e)-L\le (g,\xi_g)\le d_\Ga(g,e)$$
For all $\eta\in \dd \Ga$ one has 
$$(\xi_g,\eta)\ge \min\{(\xi_g,g),(g,\eta)\}-C\ge d_\Ga(g,e)-L-C$$ 
where $C$ comes from Equation \ref{eq:gromov:product} and thus
$$\int_{\dd \Ga} e^{-\e\ga(g,\eta)}d\nu(\eta)\prec \int_{\dd \Ga} e^{-\e\ga(\xi_g,\eta)}d\nu(\eta)\le \sup_{\xi\in \dd \Ga}\int_{\dd \Ga} e^{\e\ga(\xi,\eta)}d\nu(\eta) $$ 
\end{proof}

\subsection{Equidistribution and counting}\hfill\break
The main results of this subsection are Proposition \ref{prop:equid} and Lemma \ref{lem:tech:count} that are used in Section \ref{sec:affine} to prove the cocycle equidistribution.
Proposition \ref{prop:equid} is consequence of \cite[Thm. 3.2]{MR3939562} but as it can be proved using elementary arguments we decided to include a proof.\\

In this subsection $(\Ga,d_\Ga)$ is assumed to be $K$-roughly geodesic for a particular $K\ge0$.
We denote
$$S(R)\df\{g\in \Ga\,|\,R-K\le d_\Ga(g,e)<R+K\}$$
Given $g\in \Ga$ and $M\ge 0$ we define
$$U(g,M)\df\{\eta\in\ol{\Ga}:(g,\eta)\ge d_\Ga(g,e)-M\}$$
and 
$$O(g,M)\df\{\eta\in\dd \Ga:(g,\eta)\ge d_\Ga(g,e)-M\}=U(g,M)\cap \dd \Ga$$

\begin{prop}\label{prop:equid}
There exists a sequence of probabilities $(\nu_n)_n$ on $\Ga$ supported on $S(Kn)$ with $\nu_n(g)\prec e^{-\gd Kn}$ and
$$\sum_\Ga f(g)\nu_n(g)\xrightarrow{n} \int_{\dd \Ga}f(\xi)d\nu(\xi)$$
for all $f\in C^0(\ol{ \Ga})$.
In other words $(\nu_n)_n$ converges to $\nu$ in probability on $\ol{\Ga}$.
\end{prop}

In order to prove Proposition \ref{prop:equid} we need the following Lemma's:

\begin{lem}\label{lem:cv:labd}
There exists $R_0>0$, for all $R\ge R_0$, for all $\xi\in\dd \Ga$, there exists $g\in S(Kn)$ such that $\xi\in O(g,R)$.
\end{lem}
\begin{proof}
Let $\gamma$ be a $K$-rough geodesic between $e\in\Ga$ and $\xi\in \dd \Ga$.
One has  $h\df \gamma(Kn)\in S(Kn)\subset \Ga$ and 
$(h,\xi)=\lim_t (h,\gamma(t))\ge Kn-\frac{K}{2}$.
In other words $\xi\in O(g,R)$ for $R\ge R_0\df2K$.
\end{proof}

\begin{lem}\label{lem:equi:c0}
Let $L$ as in Lemma \ref{lem:proj:inf}.
For all $f\in C^0(\ol{\Ga})$ and $\e>0$, there exists $N$, for all $n\ge N$, $g\in S(Kn)$ and $\xi\in O(g,L)$ one has $|f(g)-f(\xi)|\le \e$.
\end{lem}
\begin{proof}
As a consequence of Lemma \ref{lem:cv:labd} $(U(g,R))_{g\in \Ga}$ forms a fundamental system of neighbourhoods of $\dd \Ga$ in $\ol{\Ga}$ for any $R\ge R_0$ (see \cite[Chap. 2]{Kapovich:aa}).
Since $f$ is continuous, for all $\xi\in \dd\Ga$, there exists $g(\xi)\in \Ga$, for all $z\in U(g(\xi),R_0+C)$, $|f(\xi)-f(z)|\le \e$, where $C$ comes from Equation \ref{eq:gromov:product}.
As $\dd \Ga $ is compact, there exists $(g(\xi_i))_{i=1,\dots k}$ such that $\dd \Ga\subset \bigcup_{i=1}^kU(g(\xi_i),R_0)$.
Let $N\ge \max_{i=1,\dots, k}d_\Ga(g(\xi_i),e)+L+K$.
Then for $n\ge N$, $g\in S(Kn)$, $\xi_g\in O(g,L)$ (see Lemma \ref{lem:proj:inf}) and $i$ with $\xi_g\in U(g(\xi_i),R_0)$ one has 
\begin{align*}
(g,g(\xi_i))&\ge \min\{(\xi_g,g(\xi_i)),(\xi_g,g)\}-C\\
&\ge \min\{d_\Ga(g(\xi_i),e)-R_0,d_\Ga(g,e)-L\}-C\\
&\ge d_\Ga(g(\xi_i),e)-R_0-C
\end{align*}
In other words $g\in U(g(\xi_i),R_0+C)$ and thus $|f(g)-f(\xi)|\le \e$.
\end{proof}

\begin{proof}[Proof of Proposition \ref{lem:cv:labd}]
Let $R_0$ as in Lemma \ref{lem:cv:labd}.\\
Write $S(Kn)=\{g_1,\dots, g_{|S_K(n)|}\}$ and define $V_n(1)\df O(g_1,R_0)$ and $V_n(i)\df O(g_i,R_0)\setminus \bigcup_{j<i} O(g_j,R_0)$ for $i>1$.\\
The probability measure $\nu_n$ is defined as:
$$\nu_n=\sum_{i=1}^k\nu(V_n(i))\text{Dir}_{g_i}$$
where $\text{Dir}_{g_i}$ denote the Dirac mass at $g_i$.
Then $\nu_n(g_i)\le \nu(O(g_i,R_0))\prec_\text{Shadow Lem.} e^{-\gd Kn}$
and given $f\in C^0(\dd \Ga)$ and $\e>0$ it follows from Lemma \ref{lem:equi:c0} that for $n$ large enough:
\begin{align*}
|\nu(f)-\nu_n(f)|&=\int_{\dd \Ga} f(\xi) d\nu(\xi)-\sum_{i=1}^{|S_R(n)|}f(g_i)\nu({U_k(n)})\\
&\le\sum_{i=1}^{|S_R(n)|}\int_{V_n(i)}|f(\xi)-f(g_i)| d\nu(\xi)
\le \e
\end{align*}
as $V_n(i)\subset O(g_i,R_0)$.

\end{proof}

\begin{lem}\label{lem:tech:count}
Let $(\nu_n)_n$ be as in Proposition \ref{prop:equid}.
Then $\sup_\xi \sum_{g\in\Ga}(\xi,g)\nu_n(g)$ is finite.
\end{lem}

In order to prove Proposition \ref{lem:tech:count} we need the following Lemma:
\begin{lem}\label{lem:counting:pro}
Let $M\ge 0$ and $n\in \BN$ such that $K(n-1)-R_0\ge M$ with $R_0$ as in Lemma \ref{lem:cv:labd}.
Then
$$|\{g\in S(Kn):(\xi,g)\ge M\}|\asymp e^{\gd Rn}e^{-\gd M}$$
\end{lem}

\begin{proof}
Let $\gamma$ be a $K$-rough geodesic between $e\in \Ga$ and $\xi\in \dd \Ga$ and $g\in \{h\in S(Kn):(\xi,h)\ge M\}$.
Then 
$$(g,\gamma(M))\ge \min\{(g,\xi),(\xi,\gamma(M))\}-C\ge M-C-K$$
and thus 
$$d_\Ga(g, \gamma(Kn))=d_\Ga(g, e)+d_\Ga(\gamma(M),e)-2(g,\gamma(Kn))
\le Kn-M+4K+C$$
,i.e $\{h\in S(Kn):(\xi,h)\ge M\}\subset B(\gamma(M),Kn-M+4K+C)$.
It follows that 
$$|\{h\in S(Kn):(\xi,h)\ge M\}|\prec e^{-\gd M}e^{\gd Kn}$$
as $|B_\Ga(g,r)|\asymp e^{\gd r}$ for all $g\in \Ga$ (see \cite[Thm. 7.2]{MR1214072}).

On the other hand 
\begin{align*}
e^{-\gd M}&\asymp_\text{Ahlfors reg.}\nu(\{\eta:(\xi,\eta)\ge M\})\\
&\le\nu(\bigcup_{g\in\{h\in S(Kn):(\xi,h)\ge M\}}O(g,R_0))\\
&\le |\{h\in S(Kn):(\xi,h)\ge M\}|e^{-\gd Kn}
\end{align*}
\end{proof}

\begin{proof}[Proof of Lemma \ref{lem:tech:count}]
Given $\xi\in \dd \Ga$ one has:
\begin{align*}
\sum_\Ga(\xi,g)\nu_n(g)&=\int_0^{d_\Ga(g,e)}\nu_n(\{(\xi,.)>r\}dr\\
&\prec \int_0^{d_\Ga(g,e)}|\{g\in S(Kn):(\xi,g)>r\}|e^{-\gd Rn} dr\\
&\asymp_{*} \int_0^{d_\Ga(g,e)}e^{-\gd r} dr\le \frac{1}{\gd}
\end{align*}
where $*$ follows from Lemma \ref{lem:counting:pro}.
\end{proof}


\section{Construction of the special representations}\label{sec:sperep}

The Hausdorff dimension of $(\dd \Ga,d)$  can be denoted $D$ or $D_\e$ to insist on the dependence on the visual parameter $\e$ (see Definition \ref{def:str:hyp}).\\
For $s\in \BR$ we denote $\pi_s$ the representation by bounded operators on $L^2(\nu)$ defined as $\pi_s(g)\vp\df e^{-sDb_\xi(g,e)}\vp(g^{-1}\xi)$ for $\vp\in L^2(\nu)$ and $\xi\in \dd \Ga$.

In this section we construct a unitary structure on the space $(L^2_0(\nu),\pi_1)$ and its dual $(L^2(\nu)/\BC,\pi_0)$  and prove preliminary results about their relations with the boundary complementary series introduced in \cite{Boucher:2020aa}.\\

A continuous map $f$ on $\ol{\Ga}$ is called {\it visually $\ga$-H\"older} for $\ga>0$ if there exists $C\ge0$ such that $|f(x)-f(y)|\le Ce^{-\e\ga(x,y)}$ for all $x,y\in \ol{\Ga}$.
In particular $f$ is continuous at every point of $\dd \Ga$ and $f|_{\dd\Ga}$ is $\ga$-H\"older for the visual metric $d$ on $\dd X$.
We denote $C_\sphericalangle^\ga(\ol{\Ga})$ the space of visually $\ga$-H\"older functions on $\ol{\Ga}$.

\begin{lem}\label{lem:int}
The operator, $I'$, defined as:
$$I':L^2(\nu)\rightarrow C_\sphericalangle^{\ga}(\ol{\Ga});\quad \vp\mapsto [z\mapsto \int_{\dd \Ga}\vp(\xi)(\xi,z)d\nu(\xi)]$$
is well defined and bounded for all $0<\ga\le \min\{1,D_\e/3\}$.
Moreover $I'$ is bounded and self-adjoint on $L^2(\nu)$.
\end{lem}
\begin{proof}
Let $0<\ga'\le \e$ such that $2\ga' <\gd $. 
Given $x,y\in \ol{\Ga}$ we define $U(x,y)\df\{\xi\in\dd\Ga:(\xi,x)<(\xi,y)\}$.
Then 
\begin{align*}
|I'(\vp)(x)-I'(\vp)(y)|
&\le \int |\vp(\xi)||(\xi,x)-(\xi,y)|d\nu(\xi)\\
&\le \frac{1}{\ga'}\int |\vp(\xi)||\int_{e^{-\ga'(\xi,x)}}^{e^{-\ga'(\xi,y)}}\frac{dt}{t}|d\nu(\xi)\\
&\le \frac{1}{\ga'}\int_{U(x,y)} \frac{|\vp(\xi)|}{e^{-\ga'(\xi,x)}}|e^{-\ga'(\xi,x)}-e^{-\ga'(\xi,y)}|d\nu(\xi)\\
&+\frac{1}{\ga'}\int_{U(x,y)^c} \frac{|\vp(\xi)|}{e^{-\ga'(\xi,y)}}|e^{-\ga'(\xi,x)}-e^{-\ga'(\xi,y)}|d\nu(\xi)\\
&\le_{\text{as $\ga'\le \e$}} \frac{1}{\ga'}\int_{\dd \Ga} \frac{|\vp(\xi)|}{e^{-\ga'(\xi,x)}}+\frac{|\vp(\xi)|}{e^{-\ga'(\xi,y)}}d\nu(\xi)e^{-\ga'(x,y)}
\end{align*}
For all $z\in\ol{\Ga}$:
\begin{align*}
 \int_{\dd \Ga} \frac{|\vp(\xi)|}{e^{-\ga'(\xi,z)}}d\nu(\xi)\le (\int_{\dd \Ga} |\vp(\xi)|^2d\nu(\xi))^\half(\int_{\dd \Ga} e^{2\ga'(\xi,z)}d\nu(\xi))^\half\le_\text{Cor. \ref{cor:unif:power}} C\|\vp\|_2
\end{align*}
and thus $I'(\vp)$ is $\ga'/\e$-Holder as $\dd \Ga$ has finite diameter.
The boundedness on $L^2(\nu)$ follows from the continuity of the inclusion $C^\ga(\dd \Ga)\subset L^2(\nu)$ for any $0<\ga$.
\end{proof}

The operator $I'$ can be seen as an intertwiner:
\begin{prop}
The operator $I'$ defines a bounded intertwiner between $(L^2_0(\nu),\pi_1)$ and $(L^2(\nu)/\BC,\pi_0)$.
In particular $I'$ induces an $\Ga$-invariant quadratic form on $L^2_0(\nu)$ given by:
$$B'(\vp,\psi)\df\int_{\dd \Ga\times\dd \Ga}\vp(\xi)\ol{\vp}(\eta)(\xi,\eta)d\nu(\xi)d\nu(\eta)$$
\end{prop}
\begin{proof}
Given any $\vp\in L^2_0(\nu)$ one has:
\begin{align*}
I'(\pi_1(g)\vp)(\eta)&=\int_{\dd \Ga}e^{-Db_\xi(g,e)}\vp(g^{-1}\xi)(\xi,\eta)d\nu(\xi)\\
&=\int_{\dd \Ga}\vp(\xi)(g\xi,\eta)d\nu(\xi)\\
&=\int_{\dd \Ga}\vp(\xi)(\xi,g^{-1}\eta)d\nu(\xi)+\half\int_{\dd \Ga}\vp(\xi)b_\xi(g^{-1},e)d\nu(\xi)\\
&=\pi_0(g)I'(\vp)(\eta)+\half\int_{\dd \Ga}\vp(\xi)b_\xi(g^{-1},e)d\nu(\xi)
\end{align*}
for almost every $\eta\in\dd \Ga$, where the third equality uses Equation (\ref{eq:conf:fond}) and the fact that $\vp$ has zero average .
In other words $I'(\pi_1(g)\vp)-\pi_0(g)I'(\vp)\in\BC$ and thus $I'(\pi_1(g)\vp)=\pi_0(g)I'(\vp)$ in $L^2(\nu)/\BC$.\\
The rest follows from the fact that $\pi_0(g)^*=\pi_1(g^{-1})$ for all $g\in \Ga$.
\end{proof}

The Knapp-Stein operators are defined for $s> \half$ as:
$$I_s(\vp)\df\int_{\dd \Ga}\frac{\vp(\xi)}{d^{2(1-s)D}(\xi,\eta)}d\nu(\xi)=\int_{\dd \Ga}\vp(\xi)e^{2(1-s)\gd(\xi,\eta)}d\nu(\xi)$$
(see \cite{Boucher:2020aa} for more details).
The following result establishes a relation between the spectrum of the operator $I'$ and the family of Knapp-Stein operators $(I_s)_s\in(\half,1]$ on $L^2(\nu)$:
\begin{prop}\label{prop:pos'}
If there exists $s_0\in (\half,1)$ such that the operators $(I_s)_{s\in(s_0,1)}$ are all positive definite on $L^2(\nu)$, then the operator $I'$ is positive definite on $L^2_0(\nu)$.
\end{prop}
In order to prove Proposition \ref{prop:pos'} we need the following lemma:
\begin{lem}\label{lem:estim:log}
For $n\ge0$  the operator:
$$K_n:L^2(\nu)\rightarrow L^2(\nu);\quad \vp\mapsto \int_{\dd \Ga}\vp(\xi)(\xi,\eta)^nd\nu(\xi)$$
is well defined and satisfies $\|K_n\|_\text{op}\prec \frac{n!}{\gd^n}$
\end{lem}
\begin{proof}
Using Schur's test it is enough to show that 
$$\sup_{\xi\in\dd \Ga}\int_{\dd \Ga} (\xi,\eta)^nd\nu(\eta)\prec \frac{n!}{\gd^n}$$
for $n\ge 0$.
The case $n=0$ is obvious so let us assume $n\ge 1$.
In this case:
\begin{align*}
\int_{\dd \Ga} (\xi,\eta)^nd\nu(\eta)&=n\int_0^\infty r^{n-1}\nu(\{(\xi,.)>r\})dr\\
&=n\int_0^\infty r^{n-1}\nu(\{d(\xi,.)<e^{-\e r}\})dr\\
&\asymp_{\gd=D_\e\e} n\int_0^\infty r^{n-1}e^{-\gd r}dr=\frac{n!}{\gd^n}
\end{align*}

where the last equality follows from:
\begin{align*}
\int_0^\infty r^{k}e^{-\ga r}dr&=\frac{-1}{\ga}[r^{k}e^{-\ga r}]_0^\infty+\frac{k}{\ga}\int_0^\infty r^{k-1}e^{-\ga r}dr\\
&=\frac{k}{\ga}\frac{(k-1)!}{\ga^{k-1}}\int_0^\infty e^{-\gd r}dr
=\frac{k!}{\ga^{k+1}}
\end{align*}
for all $k\ge 1$.

\end{proof}

\begin{proof}
As proved in Lemma \ref{lem:estim:log}, the sequence of operators
$$K_n:L^2(\nu)\rightarrow L^2(\nu);\quad \vp\mapsto \int_{\dd \Ga}\vp(\xi)(\xi,\eta)^nd\nu(\xi)$$
is well defined for $n\ge0$ and satisfies $\|K_n\|_\text{op}\prec\frac{n!}{\gd^n}$.
Observe that $K_0(\vp)=\nu(\vp)$.
Since 
$$d^{-2(1-s)D}(\xi,\eta)=e^{2(1-s)\gd(\xi,\eta)}=\sum_{k\ge0}\frac{(2(1-s)\gd)^k}{k!}(\xi,\eta)^k$$ 
on $\dd \Ga$, one has $I_s=\sum_n\frac{(2(1-s)\gd)^n}{n!}K_n$
as a limit of operators for $\half< s$.

Assuming $\vp\in L^2_0(\nu)$ and $s\in (s_0,1)$:

$$\frac{(I_s(\vp),\vp)}{2(1-s)\gd}=(I'(\vp),\vp)+2(1-s)\gd\sum_{n\ge2}\frac{(2(1-s)\gd)^{n-2}}{n!}(K_n(\vp),\vp)$$
and as 
\begin{align*}
|\sum_{n\ge2}\frac{(2(1-s)\gd)^{n-2}}{n!}(K_n(\vp),\vp)|&
\le \frac{\|\vp\|^2}{(2s-1)\gd^{2}}
\end{align*}
it follows that $(I'(\vp),\vp)\ge 0$.

\end{proof}

In the rest $H_1'$ stands for the unitary representation obtained from the Hilbert completion of $L^2_0(\nu)$ with respect to the $I'$.

\begin{cor}\label{cor:dual'}
There exists a Hilbert structure on $I'[L^2_0(\nu)]/\BC\subset L^2(\nu)/\BC$, denoted $H_0'$ that makes $\pi_0$ unitary such that the intertwiner $I':L^2_0(\nu)\rightarrow I'[L^2_0(\nu)]/\BC$ extends as an isometric intertwiner between $H_1'$ and $H_0'$.
Moreover the coupling:
 $$(u,v)\in L^2(\nu)\times L^2(\nu)\mapsto \int u(\xi)\ol{v}(\xi) d\nu(\xi)$$
  extends uniquely as a $\Ga$-invariant, non-degenerated and bounded bilinear form between $H'_1$ and $H'_0$.
\end{cor}
\begin{proof}
This is direct application of Appendix \ref{app:banach} Proposition \ref{prop:ext:equi}.
\end{proof}

We conclude this section with the following result:
\begin{lem}\label{lem:mix'}
The representations $H_1'\simeq H_0'$ are mixing, i.e their matrix coefficients vanish at infinity.
In particular $H_1'$ and $H_0'$ do not admit finite dimensional sub-representations.
\end{lem}
\begin{proof}
For all $\vp,\psi\in L^2_0(\nu)$ one has 
$$(I'(\pi_1(g_n)\vp),\psi)=(\pi_0(g_n)I'(\vp),\psi)\rightarrow I'(\vp)(\xi)({\bf 1}_{\dd \Ga},\psi)=0$$
whenever $(g_n)_n\subset \Ga$ is such that $g_n\rightarrow \xi\in\dd \Ga$ due to Lemma \ref{lem:int} (and a dominated convergence argument).
The case $u,v\in H_1'$ follows from the density of $L^2_0(\nu)$ in $H_1'$.

\end{proof}

\section{Special representations and cohomology}\label{sec:cohomology}
In this section we investigate the  cohomology of the unitary representation $H_1'\simeq H_0'$ and prove Theorem \ref{thm:main1} and \ref{thm:maincor} that will respectively follows from Corollary \ref{cor:thm2} and \ref{cor:thm1}.\\

We start by introducing some notations and recalling basic results about cohomology (see \cite{MR2415834} \cite{MR3309999} for more details):

Given a unitary representation, $(\pi,H)$, the space of cocycles for $\pi$ is defined as the subspace of maps $b:\Ga\rightarrow H$ which satisfy:
$$b(gh)=\pi(g)b(h)+b(g)$$
for all $g,h\in \Ga$ and denoted $Z^1(\Ga,\pi)$.
This space is endowed with the compact-open topology.
The subspace of co-boundary is defined as the range of operator:
$$\dd: H\rightarrow Z^1(\Ga,\pi);\quad w\mapsto \dd_w(g)=w-\pi(g)w$$
and denoted $B^1(\Ga,\pi)$.
\begin{defn}
The first cohomology and the reduced cohomology of $\Ga$ with coefficient in $\pi$ are respectively defined as the quotient spaces:
$$H^1(\Ga,\pi)\df Z^1(\Ga,\pi)/{B^1(\Ga,\pi)}$$
and
$$\ol{H}^1(\Ga,\pi)\df Z^1(\Ga,\pi)/\ol{B^1(G,\pi)}$$
where $\ol{B^1(G,\pi)}$ stands for the closer of $B^1(G,\pi)$ inside of $Z^1(G,\pi)$.
\end{defn}

Given a cocycle $b\in Z^1(\Ga,\pi)$, the formula:
$$A_{b}(g)v=\pi(g)v+b(g)$$
with $g\in \Ga$ and $v\in H$ defines an $\Ga$-affine action with linear part $\pi$ on $H$.
Conversely, any affine action with linear part $\pi$ is induced by an element of $Z^1(\Ga,\pi)$.

The following result is standard:
\begin{thm*}[\cite{MR2415834}]
Given a unitary representation $\pi$ and a cocycle $b\in Z^1(\Ga,\pi)$.
$b\in B^1(\Ga,\pi)$ if and only if its associated affine action has a fixed point or equivalently has bounded orbits.
\end{thm*}
In other words the first cohomology of $\Ga$ with coefficient in $\pi$ classifies unbounded affine actions with linear part $\pi$ up to conjugation by a affine translation.\\

In our setup the exact sequence of $\Ga$-spaces:
$$0\rightarrow L^2_0(\nu)\rightarrow L^2(\nu)\rightarrow \BC\rightarrow 0$$
induces a map:
$$c:\Ga\times \Ga\rightarrow L^2_0(\nu);\quad (x,y)\mapsto \pi_1(x){\bf1}_{\dd \Ga}-\pi_1(y){\bf1}_{\dd \Ga}$$ 
that is $\Ga$-equivariant in the sense: $c(g.x,g.y)=\pi_1(g)c(x,y)$ for all $g,x,y\in \Ga$.
In particular $g\mapsto c(g,e)=-c(e,g)$ is cocycle with coefficients in $\pi_1$ as defined above.

We investigate the properties of $c$ within the Hilbert completion, $H_1'$, of $L^2_0(\nu)$ constructed in the Section \ref{sec:sperep}.

The following relation is extensively used in the rest:
\begin{lem}\label{lem:fond'}
$$I'(c(x,y))=b_.(y,x)+\dd I'({\bf1}_{\dd\Ga})(x,y)$$
for all $x,y\in \Ga$, where $\dd I'({\bf1}_{\dd\Ga})(x,y)(\xi)\df I'({\bf1}_{\dd\Ga})(x^{-1}\xi)-I'({\bf1}_{\dd\Ga})(y^{-1}\xi)$.
In particular whenever $I'$ is positive definite on $L^2_0(\nu)$, $I'(c)$ is a cocycle in the same  cohomology class as the Busemann cocycle in $(H_0',\pi_0)$.
\end{lem}
\begin{proof}
Let $v\in L^2_0(\nu)$ and $x,y\in \Ga$.

\begin{align*}
&(I'(c(x,y)),v)=\int_{\dd \Ga\times\dd \Ga}(\xi,\eta)(e^{-Db_\xi(x,e)}-e^{-Db_\xi(y,e)})d\nu(\xi)\ol{v}(\eta)d\nu(\eta)\\
&=\int_{\dd \Ga\times\dd \Ga}(x\xi,\eta)d\nu(\xi)\ol{v}(\eta)d\nu(\eta)
-\int_{\dd \Ga\times\dd \Ga}(y\xi,\eta)d\nu(\xi)\ol{v}(\eta)d\nu(\eta)\\
&=_\text{Eq (\ref{eq:conf:fond})}\int_{\dd \Ga\times\dd \Ga}(\xi,x^{-1}\eta)+\half b_\xi(x^{-1},e)+\half b_\eta(e,x)d\nu(\xi)\ol{v}(\eta)d\nu(\eta)\\
&-\int_{\dd \Ga\times\dd \Ga}(\xi,y^{-1}\eta)+\half b_\xi(y^{-1},e)+\half b_\eta(e,y)d\nu(\xi)\ol{v}(\eta)d\nu(\eta)\\
&=_{\nu(v)=0} \half\int_{\dd \Ga}b_\eta(y,x)\ol{v}(\eta)d\nu(\eta)
+\int_{\dd \Ga}[I'({\bf1}_{\dd \Ga})(x^{-1}\eta)-I'({\bf1}_{\dd \Ga})(y^{-1}\eta)]\ol{v}(\eta)d\nu(\eta)
\end{align*}

\end{proof}

\begin{prop}\label{prop:kvg}
$$d_\Ga(x,y)\simeq (I'(c(x,y)), c(x,y))$$
for all $x,y\in \Ga$.
In particular the metric $d_\Ga$ is $\Ga$-roughly conditionally negative (see Subsection \ref{subsec:intro:res}) whenever $I'$ is positive on $L^2_0(\nu)$.
\end{prop}

\begin{proof}

\begin{align*}
\half (b(y,x),c(x,y))
&=\half\int b_{x\eta}(y,x) d\nu(\eta)+\half\int b_{y\eta}(x,y) d\nu(\eta)\\
&=d_{\Ga}(x,y)-\int (\eta,x^{-1}y)d\nu(\eta)-\int (\eta,y^{-1}x)  d\nu(\eta)
\end{align*}

Lemma \ref{lem:fond'} implies:
\begin{align*}
&(I'(c(x,y)),c(x,y))
=\half (b(y,x),c(x,y))
+(\dd I'({\bf1}_{\dd\Ga})(x,y),c(x,y))\\
&=d_{\Ga}(x,y)-\int (\eta,x^{-1}y)d\nu(\eta)-\int (\eta,y^{-1}x)  d\nu(\eta)
+(\dd I'({\bf1}_{\dd\Ga})(x,y),c(x,y))
\end{align*}

Observe that 
$$|(\dd I'({\bf1}_{\dd\Ga})(x,y),c(x,y))|\le 2\nu(I'({\bf1}_{\dd\Ga}))+2\|I'({\bf1}_{\dd\Ga})\|_\infty$$
and together with Lemma \ref{lem:estim:log} proof it follows that:
$$r(x,y)\df (\dd I'({\bf1}_{\dd\Ga}),c(x,y))-\int (\eta,x^{-1}y)d\nu(\eta)-\int (\eta,y^{-1}x)  d\nu(\eta)$$
for $x,y\in \Ga$ is real, symmetric,  $\Ga$-invariant and bounded.

\end{proof}

\begin{rem}
Proposition \ref{prop:kvg} holds independently of whether $I'$ is positive on $L^2_0(\nu)$.
\end{rem}

A consequence of Proposition \ref{prop:kvg} together with \cite[Thm. 1]{Boucher:2020aa} is the following converse of Proposition \ref{prop:pos'}:
\begin{cor}\label{cor:thm2}
The family of Knapp-Stein operators $(I_s)_{s>\half}$ over $L^2(\nu)$ are positive if and only if $I'$ is positive over $L^2_0(\nu)$.
In this case the unitary representation $(H_1',\pi_1)$ is well defined and satisfies $H^1(H_1',\pi_1)\neq0$.
\end{cor}

In other words the operators $(I_s)_{s>\half}$ over $L^2(\dd X, \nu)$ are positive if and only if $d_\Ga$ is $\Ga$-roughly conditionally negative. In particular \cite[Thm. 1]{Boucher:2020aa} is tight.

\subsection{Special representations and reduced cohomology}$ $\\
Before investigating the reduced cohomology of $(H'_1,\pi_1)$ we recall some results about harmonic cocycles (see \cite{MR3831031} \cite{MR3854902} \cite{MR3711276} for more).

We call a probability measure $\beta$ on $\Ga$ {\it admissible} if it is finitely supported, symmetric and such that the $\gb$-random walk over $(\Ga,d_\Ga)$ has non-zero drift.
Observe that because $\Ga$ is assumed to be non-elementary, any probability supported on a generating set has non-zero drift \cite{MR1743100}.\\
If $(X_n)_n$ denotes a $\gb$-random walk on $(\Ga,d_\Ga)$ the last assumption is equivalent to say that the random variable $Z_n=\frac{d_\Ga(X_n,e)}{n}$ converges in $L^1(\gO,\BP)$ to a non-zero constant denoted $\ell$ \cite{MR1743100}.

Given a unitary representation $\pi$ and a admissible probability $\gb$. 
A cocycle with coefficients in $\pi$, $c_\pi$, is called $\beta$-harmonic if it satisfies:
$$\sum_{g\in\Ga} c_\pi(g)\beta(g)=0$$ or equivalently 
$$\sum_{g\in\Ga} c_\pi(hg)\beta(g)=c_\pi(h)$$ 
for all $h\in \Ga$.

The Hilbert norm  on the space of cocycles $Z^1(\Ga,\pi)\subset \text{Map}(\Ga,\mathcal{H}_\pi)$ defined as:
$$\|c_\pi\|^2_{\beta}\df\sum_{g\in\Ga}\|c_\pi(g)\|_\pi^2\beta(g)$$
induces an equivalent topology as the compact-open topology \cite{MR3854902} and the following $\|\,.\,\|_\gb$-orthogonal decomposition holds:
$$Z^1(\Ga,\pi)=B^1(\Ga,\pi)^\perp\oplus \ol{B^1(\Ga,\pi)}$$
Since
$$\sum_{g\in\Ga}(c_\pi(g),\xi-\pi(g)\xi)\beta(g)=2(\sum_{g\in\Ga} c_\pi(g)\beta(g),\xi)$$
for all $b\in Z^1(\Ga,\pi)$, we have the isometric identification:
$$B^1(\Ga,\pi)^\perp=Z_{\text{Harm}}^1(\Ga,\pi)$$
where $Z_{\text{Harm}}^1(\Ga,\pi)$ stands for the closed space of $\beta$-harmonic cocycles and thus:
$$Z_{\text{Harm}}^1(\Ga,\pi)\simeq \ol{H}^1(\Ga,\pi)=Z^1(\Ga,\pi)/{\ol{B^1(\Ga,\pi)}}$$

The following criterion was proved by Erschler and Ozawa in \cite{MR3854902}:
\begin{lem}
Given a unitary representation $\pi$ and $c_\pi\in Z^1(\Ga,\pi)$:
$$\lim_n\frac{\BE_\gb[\|c_\pi(X_n)\|_\pi^2]}{n}=\|\mathcal{P}_\text{Harm}[c_\pi]\|_\beta^2$$
where $\mathcal{P}_\text{Harm}:Z^1(\Ga,\pi)\rightarrow B^1(\Ga,\pi)^\perp=Z_{\text{Harm}}^1(\Ga,\pi)$ stands for the orthogonal projection over the $\gb$-harmonic part of $c_\pi$ and $(X_n)_n$ for a $\beta$-random walk on $\Ga$.
\end{lem}

\begin{cor}\label{cor:thm1}
Let $\beta$ be an admissible probability on $\Ga$ and $\Lambda\subset \Ga$ the subgroup generated by its support.
The reduced cohomology class $c|_\Lambda$ is non-trivial in $\ol{H}^1(\Lambda,\pi_1|_{\Lambda})$.
\end{cor}
\begin{proof}
As a consequence of  Kingman's subadditive theorem (see \cite{MR1743100}) together with Proposition \ref{prop:kvg} one has:
\begin{align*}
\lim_n\frac{1}{n}\BE_\gb[\|c(X_n)\|_{H_1'}^2]&=\lim_n\frac{1}{n}\BE_\gb[d_\Ga(e,X_n)]\\
&=\lim_n\frac{\BE_\gb[d_\Ga(e,X_n)]}{n}=\ell(\beta)
\end{align*}
where $\ell(\beta)$ stands for the drift of $\beta$ on $(\Ga,d_\Ga)$.
It follows that $\|\mathcal{P}_\text{Harm}[c]\|_\beta^2=\ell(\beta)>0$ and as $\ol{H}^1(\Lambda,\pi|_{\Lambda})\simeq Z_{\text{Harm}}^1(\Lambda,\pi|_{\Lambda})$ the cocycle $c$ has non trivial  reduced cohomology class.
\end{proof}

\section{Irreducibility of the special affine actions}\label{sec:affine}
We denote $c(g)\df c(g,e)$ and $b(g)=b_.(g,e)$ the Buseman cocycle for $g\in \Ga$.
In this section we investigate the irreducibility of the affine action $(H_1',\pi_1,c)\simeq(H_0',\pi_0,b) $ and prove Theorem \ref{thm:main2} and \ref{thm:main3}.

We use the following criterion proved by Bekka, Pillon and Valette in \cite{MR3549540}:
\begin{prop}[B-P-V \cite{MR3549540}]
An isometric affine action $(H,\pi,c_\pi)$ is irreducible if and only if for every $w\in H$, $\ol{{\bf Span}\{c_\pi(g)+\dd_w(g)\,|\,g\in \Ga\}}=H$ or equivalently if and only if the affine commutant of the $\Ga$-affine action is reduced to translations by $\pi$-fixed vectors.
\end{prop}

Since $C(\dd \Ga)\cap L^2_0(\nu)$ is dense in $H'_1$ by construction, Theorem \ref{thm:main2} together with the above B-P-V criterion imply the irreducibility of the affine action $(H_1',\pi_1,c)$. 
Since $H'_1$ does not have any fixed vector due to Lemma \ref{lem:mix'}, Theorem \ref{thm:main3} also follows. \\

The main step consists to reduce the proof of Theorem \ref{thm:main2} to the case where $u\in L^2_0(\nu)$ and $w=0$. 
This will be a consequence of the following technical lemmas:

\begin{lem}\label{lem:alphabond}
Let $(\nu_n)_n$ be as in Proposition \ref{prop:equid}. The map:
$$h\in \Ga\mapsto\sum_\Ga |(b(g),c(h))| \nu_{n}(g)$$
is uniformly bounded on $\Ga$.
\end{lem}
\begin{proof}
Observe that 
$$(b(g),\vp)=2\int_{\dd \Ga} (g,\xi)\vp(\xi)d\nu(\xi)$$ 
whenever $\vp\in L^2_0(\nu)$ for all $g\in \Ga$.
It follows:
\begin{align*}
&\sum_\Ga|(b(g),c(h))| \nu_n(g)
=2\sum_\Ga|\int_{\dd \Ga}(\xi,g)(e^{-Db_\xi(h,e)}-1)d\nu(\xi)|\nu_n(g)\\
&\le 2\int_{\dd \Ga}[\sum_\Ga(\xi,g)\nu_n(g)]e^{-Db_\xi(h,e)}d\nu(\xi)
+2\sum_\Ga\int_{\dd \Ga}(\xi,g)d\nu(\xi)\nu_n(g)\\
&\le 2\|\sum_\Ga(.,g)\nu_n(g)\|_\infty+2\|\int_{\dd \Ga}(\xi,.)d\nu(\xi)\|_\infty
\end{align*}
The result follows from Lemma \ref{lem:tech:count} and the proof of Lemma \ref{lem:estim:log}. 
\end{proof}

\begin{cor}\label{cor:bound1}
Let $(\nu_n)_n$ be as in Proposition \ref{prop:equid}.
Given $w\in H_1'$ the family of operators:
$$A_n(w):C^0(\ol{\Ga})\rightarrow H_1';\quad f\mapsto\sum_\Ga f(g)\nu_{n}(g)(c(g)+\dd_w(g))$$
is uniformly bounded on $n\in\BN$.
\end{cor}
\begin{proof}
Since the co-boundary $\dd_w$ is bounded we can assume $w=0$.

Observe that:
\begin{align*}
\|&\sum_\Ga f(g)\nu_n(g)c(g)\|^2=\sum_{\Ga\times \Ga}f(g)f(h)(I'(c(g)),c(h))\nu_n(g)\nu_n(h)\\
&\le\|f\|_\infty^2\sum_{\Ga\times \Ga}|(I'(c(g)),c(h))|\nu_n(g)\nu_n(h)\\
&\le \|f\|_\infty^2\sum_{\Ga\times \Ga}|(b(g),c(h))|\nu_n(g)\nu_n(h)\\
&+\|f\|_\infty^2\sum_{\Ga\times \Ga}|(\dd I'({\bf1}_{\dd \Ga})(g,e),c(h))|\nu_n(g)\nu_n(h)
\end{align*}
which is uniformly bounded on $n$ due to Lemma \ref{lem:alphabond} and the fact that $\sup_g\|\dd I'({\bf1}_{\dd \Ga})(g,e)\|_\infty\le 2\|I'({\bf1}_{\dd \Ga})\|_\infty$.
\end{proof}

As $(\pi_1,H_1')$ is mixing:
\begin{lem}\label{lem:affmix}
Let $(\nu_n)_n$ be as in Proposition \ref{prop:equid}.
Given any $u,v\in H_1'$ and $f\in C(\ol{\Ga})$ with $\nu(f)=0$:
$$\lim_n\sum_\Ga f(g)(I'(\dd_{u}(g)),v)\nu_{n}(g)=0$$
\end{lem}
\begin{proof}
By definition:
\begin{align*}
\sum_\Ga f(g)&(I'(\dd_{u}(g)),v)\nu_{n}(g)=(I'(u),v)\sum_\Ga f(g)\nu_{n}(g)\\
&-\sum_\Ga(f(g)(I'(\pi_1(g)u),v))\nu_{n}(g)
\end{align*}
for all $u,v\in \mathcal{H}_1'$ and $f\in C(\ol{\Ga})$ with $ \nu(f)=0$.
It follows from Lemma \ref{lem:mix'} and Proposition \ref{prop:equid}:
\begin{align*}
\lim_n(I'(u),v)\sum_\Ga&f(g)\nu_{n}(g)-\sum_\Ga[f(g)(I'(\pi_1(g))u,v)]\nu_{n}(g)\\
&=(I'(u),v)\nu(f)-\nu(0)=0
\end{align*}

Indeed the function $g\mapsto f(g)(I'(\pi_1(g))u,v)$ extends continuously over the boundary by zero.
\end{proof}

We reduced Theorem \ref{thm:main2} proof to the following Proposition:
\begin{prop}\label{lem:c01}
Let $(\nu_n)_n$ be as in Proposition \ref{prop:equid}.
Given any $f\in C^0(\ol{\Ga})$ with $\nu(f)=0$ and $\vp\in  L^2_0(\nu)$:
$$\lim_n\sum_\Ga f(g)(I'(c(g)),\vp)\nu_n(g)=2(I'(f|_{\dd X}),\vp)$$
\end{prop}
Indeed, using Lemma \ref{lem:affmix}, we can assume $w=0$.
On the other hand Corollary \ref{cor:bound1} allows to extend the weak convergence of $$v_n=\sum_\Ga f(g)c(g)\nu_{n}(g)$$ with $n\in \BN$, on the dense subspace $L^2_0(\nu)$ to $H_1'$.

\begin{proof}[Proof of Proposition \ref{lem:c01}]
As consequence of Lemma \ref{lem:fond'}, we have
$$(I'(c(g)),\vp)=2I'(\vp)(g)+(\dd I({\bf1}_{\dd \Ga})(g,e),\vp)$$
for $g\in \Ga$ and $\vp\in  L^2_0(\nu)$.

It follows:
\begin{align*}
I_n&\df\sum_\Ga f(g)(I'(c(g)),\vp)\nu_{n}(g)\\
&=2\sum_\Ga f(g)I'(\vp)(g)\nu_{n}(g)\\
&+\sum_\Ga f(g)(\pi_0(g)I'({\bf1}_{\dd \Ga}),\vp)\nu_n(g)
+(I'({\bf1}_{\dd \Ga}),\vp)\sum_\Ga f(g)\nu_n(g)
\end{align*}

Observe that $g\mapsto f(g)(\pi_0(g)I'({\bf1}_{\dd \Ga}),\vp)$ extends to zero on the boundary $\dd \Ga$ on $\Ga$.
Indeed, given a proper sequence $(g_n)_n$ in $\Ga$ one can assume that $g_n\rightarrow \eta$ and $g_n'\rightarrow\eta'$ with $\eta,\eta'\in\dd\Ga$ and dominated convergence (as $I'({\bf1}_{\dd \Ga})\in L^\infty(\dd \Ga)$):
$$f(g_n)\int_{\dd \Ga}I'({\bf1}_{\dd \Ga})(g_n^{-1}\xi)\vp(\xi)d\nu(\xi)
\rightarrow f(\eta)I'({\bf1}_{\dd \Ga})(\eta')\nu(\vp)=0$$

Using Proposition \ref{prop:equid} together with Lemma \ref{lem:int} we get:
\begin{align*}
\lim_nI_n&=2(I'(f|_{\dd X}),\vp)-\nu(f)\int I'({\bf1}_{\dd \Ga})(\xi)\ol{\vp(\xi)}d\nu(\xi)\\
&=2(I'(f|_{\dd X}),\vp)
\end{align*}
for all $\vp\in L^2_0(\nu)$.
\end{proof}

\appendix
\section{Positive operators on reflexive Banach spaces.}\label{app:banach}
Let $V$ be a real separable reflexive Banach space and $T$ a bounded operator between $V$ and $V'$.\\
We call a bounded operator $T$ on $V$ {\it self-adjoint} if for all $x,y\in V$, $T(x)(y)=T(y)(x)$, in other words $T':V\simeq (V')'\rightarrow V'$ identifies with $T$.\\
\begin{rem}
Recall that if $W\subset V$ denote a closed subspace of $V$, then $W$ and $V/W$ are also reflexive.
Moreover the dual of $V/W$ is isomorphic to the annihilator of $W$ \cite{MR1070713}.
As the adjoint of $T$ takes value in the annihilator of the kernel of $T$, one can assume $T$ has trivial kernel whenever $T$ is self-adjoint.
Moreover it follows from Hahn-Banach that $T$ has dense range in this case.
\end{rem}
A self-adjoint operator $T$ is {\it positive} if the quadratic form $v\in V\mapsto Q_T(v)\df T(v)(v)$ is positive definite on $V$.\\
Given a positive operator the non-degenerated bilinear form, $Q_T$, induces a real Hilbert space, $H_T(V)$, by  completion of $V$ with respect to $Q_T$.
As $T$ is bounded the inclusion $v\in V\mapsto [v]\in H_T(V)$ is bounded with norm at most $\|T\|_\text{op}$.\\
Conversely given a (injective) bounded operator, $S$, from $V$ into a Hilbert space $H$, the operator $T\df S'S$ is positive on $V$.

\begin{rem}
Every separable Banach space embeds into a Hilbert space.
\end{rem}

\begin{prop}\label{prop:ext:noequi}
Let $T$ be a positive operator on $V$.
There exists a dense Hilbert subspace, $H_T'$, in $V'$ that contains the range of $T$ such that $T:V\rightarrow V'$ extends as an isometry between $H_T$ and $H_T'$ and duality between $V$ and $V'$  as a non-degenerated bilinear form between $H_T$ and $H_T'$.
\end{prop}
\begin{proof}
The operator $T$ is bijective between $V$ and $R\subset V'$ with $R$ the range of $T$.\\
Given $y\in R$ and $x\in V$ such that $T(x)=y$ we define $Q(y)\df T(x)(x)$.
The space $H_T'$ is obtained from the completion of $R$ with respect to $Q$.
It follows from the definition that $\|T(x):H_T'\|=\|x|H_T\|$, i.e $T$ induces an isometry between $H_T'$ and $H_T$.
For all $x\in V$ it follows from Hahn-Banach that one can find $x'\in V$ with $\|x'\|=1$ such that
$$\|T(x)|V'\|=T(x)(x')\le \sqrt{T(x)(x)}\sqrt{T(x')(x')}\le\|T\|_\text{op}\|x|H_T\|=\|T(x):H_T'\|$$
and thus $R\subset H_T'\subset V'$.
On the other hand for all $T(x')\in R$ and $x\in V$:
\begin{align*}
|T(x')(x)|^2
&\le_\text{C-S ineq.}\ol{T(x')}(x')\ol{T(x)}(x)\\
&=\|x|H_T\|^2\|T(x')|H_T'\|^2
\end{align*} 
which proves the extension of the duality.
\end{proof}

\begin{exam}\label{exam:pos:op:reflexive}
Let $(\BR^n,\gl)$ with $n\ge 1$ and $0<\ga< n$.
It follows from the Hardy-Littlewood-Sobolev inequality \cite[Thm. 4.3]{MR1817225} that 
$$I_\ga(\vp)(y)=\int_{\BR^n}f(x)\|x-y\|^{-n+\ga}d\gl(x)$$ is well defined for $f\in L^p(\gl)$ with $p=\frac{2n}{n-\ga}$ and takes value in $L^q(\gl)$ with $q=\frac{2n}{n+\ga}$.
It follows from the semi-multiplicative property: $I_\ga^2=c(\ga)I_{2\ga}$ with $0<\ga<\frac{n}{2}$ and $c(\ga)>0$ \cite[(25.38)]{MR1347689}, 
that $I_\ga$ is positive in the above sense.
In this case the Hilbert associated is the nothing the real space of $-\ga$-potential on $\BR^n$, i.e $H_{I_\ga}(L^p(\gl))\simeq H^{-\ga}(\BR^n)$.
\end{exam}

Let $(\pi,V)$ be a representation on $V$ by bounded operators and $(\pi',V')$ its dual representation.
Any positive intertwiner, $T$, between $\pi$ and $\pi'$ induces a positive $\pi$-invariant quadratic form on $V$ and therefore an isometric representation on $H_T$ which makes the inclusion $V\rightarrow H_T$ an intertwiner.

Similarly as in the non-equivariant case any bounded intertwiner between $(\pi,V)$ and an isometric representation on a Hilbert space induces a intertwiner between $(\pi,V)$ and $(\pi',V')$.

\begin{prop}\label{prop:ext:equi}
Let $T$ be a positive intertwiner for $(\pi,V)$.
There exists an isometric representation on a dense Hilbert space, $H_T'$, in $V'$ such that $T:V\rightarrow V'$ extends as an isometric intertwiner between $H_T$ and $H_T'$ and the duality bracket extends to $H_T\times H_T'$ as a $\Ga$-invariant non-degenerated bilinear form.
\end{prop}
\begin{proof}
The result essentially follows from the Proposition \ref{prop:ext:noequi} by observing that the dual structure itself induces a unitary representation for $\pi'$.
\end{proof}

\begin{exam}
Given $(\BR^n,\gl)$ with $n\ge 1$,  $0<\ga< n$ and $I_\ga$ as in Example \ref{exam:pos:op:reflexive}.
The group of Euclidean motions, $G=SO(n)\ltimes \BR^n$, of $\BR^n$ preserving the Lebesgue measure act by isometries on $L^p(\gl)$ for all $p\in [1,\infty]$ and commute with operator $I_\ga$. It induces isometric representations on $H^{-\ga}(\BR^n)$
\end{exam}

\nocite{*}
\bibliographystyle{plain}
\bibliography{complement2022}
\end{document}